\documentclass{elsart}

\usepackage{amsmath}
\usepackage{amstext}
\usepackage{amssymb}
\usepackage[english,francais]{babel}

\newtheorem{theorem}{Theorem}[section]

\newtheorem{proof}[theorem]{Proof}
\newtheorem{e-proposition}[theorem]{Proposition}

\newtheorem{e-definition}[theorem]{Definition\rm}

\def\og{\leavevmode\raise.3ex\hbox{$\scriptscriptstyle\langle\!\langle$~}}
\def\fg{\leavevmode\raise.3ex\hbox{~$\!\scriptscriptstyle\,\rangle\!\rangle$}}

\begin{document}
\centerline{}
\begin{frontmatter}

\selectlanguage{english}
\title{A symbolic approach to multiple zeta values \\
at the negative integers}

\author[authorlabel1]{Victor H. Moll},
\ead{vhm@tulane.edu}
\author[authorlabel1]{Lin Jiu}
\ead{ljiu@tulane.edu}
\author[authorlabel1,authorlabel2]{Christophe Vignat}
\ead{cvignat@tulane.edu}

\address[authorlabel1]{Department of Mathematics, Tulane University, New Orleans, USA}
\address[authorlabel2]{LSS/Supelec, Universit\'{e} Orsay Paris Sud, France}

\medskip
\begin{abstract}
\selectlanguage{english}
\vskip 0.5\baselineskip
\noindent
Symbolic computation techniques are used to derive some closed form
expressions for an analytic continuation of the Euler-Zagier zeta
function evaluated at the negative integers as recently proposed
in \cite{Sadaoui}. This approach allows to compute explicitly some
contiguity identities, recurrences on the depth of the zeta
values and generating functions.
\end{abstract}
\end{frontmatter}

\selectlanguage{english}
\section{Introduction}

The multiple zeta functions, first introduced by Euler and generalized
by D. Zagier \cite{Zagier}, appear in diverse areas
such as quantum field theory \cite{Broadhurst} and knot theory \cite{Takamuki}. These are defined by
\begin{equation}
\label{zeta}
\zeta_{r}\left(n_{1},\dots,n_{r}\right)=\sum_{0<k_{1}<\dots<k_{r}}\frac{1}{k_{1}^{n_{1}}\dots k_{r}^{n_{r}}},
\end{equation}
where $\left\{ n_{i}\right\} $ are complex values, and \eqref{zeta} converges when
the constraints
\begin{equation}
\label{eq:convergence}
Re \left(n_r\right)\ge1, \thinspace\thinspace \text{and}\thinspace\thinspace \overset{k}{\underset{j=1}{\sum}}Re \left(n_{r+1-j}\right)\ge k,\,\,\,\, 2 \le k \le r,
\end{equation}
are satisfied (see \cite{Zhao}). Their values at integer points $\mathbf{n}=\left( n_{1},\dots,n_{r}\right) $
satisfying \eqref{eq:convergence} are called \textit{multiple zeta values}. An
equivalent definition of these values is
\[
\zeta_{r}\left(n_{1},\dots,n_{r}\right)=\sum_{k_{1}>0,\dots,k_{r}>0}\frac{1}{k_{1}^{n_{1}}\left(k_{1}+k_{2}\right)^{n_{2}}\dots\left(k_{1}+\dots+k_{r}\right)^{n_{r}}}.
\]
The sum of the exponents $n_{1}+\dots+n_{r}$ is called the \textit{weight}
of the zeta value, and the number $r$ of these exponents is called
its \textit{depth}.

Following the result by Zhao \cite{Zhao} that the multiple zeta function
has an analytic continuation to the whole space $\mathbb{C}^{r},$
several authors have recently proposed different analytic continuations
based on a variety of approaches: Akiyama et al. \cite{Akiyama} used
the Euler-Maclaurin summation formula and Matsumoto \cite{Matsumoto}
the Mellin-Barnes integral formula.

B. Sadaoui \cite{Sadaoui} provided recently such analytic continuation based on Raabe's identity, which links
the multiple integral
\[
Y_{\mathbf{a}}\left(\mathbf{n}\right)=\int_{\left[1,+\infty\right)^{r}}\frac{d\mathbf{x}}{\left(x_{1}+a_{1}\right)^{n_{1}}\left(x_{1}+a_{1}+x_{2}+a_{2}\right)^{n_{2}}\dots\left(x_{1}+a_{1}+\dots+x_{r}+a_{r}\right)^{n_{r}}}
\]
to the multiple zeta function
\[
Z\left(\mathbf{n,}\mathbf{z}\right)=\sum_{k_{1}\ge1,\dots,k_{r}\ge1}\frac{1}{\left(k_{1}+z_{1}\right)^{n_{1}}\left(k_{1}+z_{1}+k_{2}+z_{2}\right)^{n_{2}}\dots\left(k_{1}+z_{1}+\dots+k_{r}+z_{r}\right)^{n_{r}}}
\]
by
\[
Y_{\mathbf{0}}\left(\mathbf{n}\right)=\int_{\left[0,1\right]^{r}}Z\left(\mathbf{n},\mathbf{z}\right)d\mathbf{z}.
\]
B. Sadaoui uses a classical inversion argument to obtain an analytic continuation
of the multiple zeta function defined at negative integer arguments
$\mathbf{-n}=\left(-n_{1},\dots,-n_{r}\right).$ The argument uses the following three steps:\\
- the integral $Y_{\mathbf{a}}\left(\mathbf{n}\right)$ is computed for values of $n_{1},\dots,n_{r}$ that satisfy the convergence conditions (\ref{eq:convergence}),\\
- the values $\mathbf{n}$ are replaced by $-\mathbf{n}$ in this result: it is then shown that $Y_{\mathbf{a}}\left(-\mathbf{n}\right)$ is a polynomial in the variable $\mathbf{a},$\\
- the variables $\mathbf{a}=\left(a_{1},\dots,a_{r}\right)$ are replaced
by $\left(\mathcal{B}_{1},\dots,\mathcal{B}_{r}\right),$ and each Bernoulli symbol $\mathcal{B}_{k}$ satisfies the two evaluation rules:\\
\textbf{evaluation rule 1}: each power $\mathcal{B}_{k}^{p}$ of the Bernoulli
symbol $\mathcal{B}_{k}$ should be evaluated as
\begin{equation}
\mathcal{B}_{k}^{p}\to B_{p},\label{eq:rule1}
\end{equation}
the $p-$th Bernoulli number\\
\textbf{evaluation rule 2}: for two different symbols $\mathcal{B}_{k}$ and
$\mathcal{B}_{l},\thinspace\thinspace k\ne l,$ the product $\mathcal{B}_{k}^{p}\mathcal{B}_{l}^{q}$
is evaluated as
\begin{equation}
\mathcal{B}_{k}^{p}\mathcal{B}_{l}^{q}\to B_{p}B_{q},\label{eq:rule2}
\end{equation}
the product of the Bernoulli numbers $B_{p}$ and $B_{q}.$ If $k=l,$
the first rule applies to give the evaluation
\[
\mathcal{B}_{k}^{p}\mathcal{B}_{k}^{q}\to B_{p+q}.
\]

\noindent
\textbf{Example 1}.
An  example of depth $2$, appearing in \cite{Sadaoui}, is now computed using the rules above. The integral $Y_{\mathbf{a}}\left(n_{1},n_{2}\right)$
is explicitly computed and, replacing $\left(n_{1},n_{2}\right)$
by $\left(-n_{1},-n_{2}\right)$ gives
\[
Y_{a_{1},a_{2}}\left(-n_{1},-n_{2}\right)=\frac{1}{n_{2}+1}\sum_{k_{2}=0}^{n_{2}+1}\sum_{l_{1}=0}^{n_{1}+n_{2}+2-k_{2}}\sum_{l_{2}=0}^{k_{2}}\frac{\binom{n_{2}+1}{k_{2}}\binom{n_{1}+n_{2}+2-k_{2}}{l_{1}}\binom{k_{2}}{l_{2}}}{n_{1}+n_{2}+2-k_{2}}a_{1}^{l_{1}}a_{2}^{l_{2}}.
\]
Then substituting the variables $a_{1}$ and $a_{2}$ with the Bernoulli
symbols $\mathcal{B}_{1}$ and $\mathcal{B}_{2}$ gives
\[
\zeta_{2}\left(-n_{1},-n_{2}\right)=\frac{1}{n_{2}+1}\sum_{k_{2}=0}^{n_{2}+1}\sum_{l_{1}=0}^{n_{1}+n_{2}+2-k_{2}}\sum_{l_{2}=0}^{k_{2}}\frac{\binom{n_{2}+1}{k_{2}}\binom{n_{1}+n_{2}+2-k_{2}}{l_{1}}\binom{k_{2}}{l_{2}}}{n_{1}+n_{2}+2-k_{2}}\mathcal{B}^{l_{1}}\mathcal{B}^{l_{2}}.
\]
Using the evaluation rules (\ref{eq:rule1}) and (\ref{eq:rule2})
for the Bernoulli symbols, the multiple zeta value
of depth $2$ at $\left(-n_{1},-n_{2}\right)$ is
\[
\zeta_{2}\left(-n_{1},-n_{2}\right)=\frac{1}{n_{2}+1}\sum_{k_{2}=0}^{n_{2}+1}\sum_{l_{1}=0}^{n_{1}+n_{2}+2-k_{2}}\sum_{l_{2}=0}^{k_{2}}\frac{\binom{n_{2}+1}{k_{2}}\binom{n_{1}+n_{2}+2-k_{2}}{l_{1}}\binom{k_{2}}{l_{2}}}{n_{1}+n_{2}+2-k_{2}}B_{l_{1}}B_{l_{2}}.
\]

The general case is given in \cite[eq. (4.10)]{Sadaoui} as the $\left(2r-1\right)-$fold
sum \footnote{This corrects a typo in \cite[eq. (4.10)]{Sadaoui}}
\begin{eqnarray}
\zeta_{r}\left(-n_{1},\dots,-n_{r}\right) & = & \left(-1\right)^{r}\sum_{k_{2},\dots,k_{r}}\frac{1}{\left(\bar{n}+r-\bar{k}\right)}\prod_{j=2}^{r}\frac{\binom{\sum_{i=j}^{r}n_{i}+r-j+1-\sum_{i=j+1}^{n}k_{i}}{k_{j}}}{\left(\sum_{i=j}^{r}n_{i}+r-j+1-\sum_{i=j+1}^{n}k_{i}\right)}\label{eq:general}\\
 &  & \times\sum_{l_{1},\dots,l_{r}}\binom{\bar{n}+r-\bar{k}}{l_{1}}\binom{k_{2}}{l_{2}}\dots\binom{k_{r}}{l_{r}}B_{l_{1}}\dots B_{l_{r}}\nonumber 
\end{eqnarray}
where $k_{2},\dots,k_{r}\ge0,$ $l_{j}\le k_{j}$
for $2\le j\le r$ and $l_{1}\le\bar{n}+r+\bar{k}$  and
\begin{equation}
\label{nbar}
\bar{n}=\sum_{j=1}^{r}n_{j},\thinspace\thinspace\bar{k}=\sum_{j=2}^{r}k_{j}.
\end{equation}

A symbolic expression for (\ref{eq:general}) is proposed here. This is used as a convenient tool to derive some specific zeta values
at negative integers, contiguity identities for the
multiple zeta functions, recursions on their depth and generating functions.

\section{Main result}

Introduce first the symbols $\mathcal{C}_{1,2,\dots,k}$ defined
recursively in terms of the Bernoulli symbols $\mathcal{B}_{1},\dots,\mathcal{B}_r$ as
\[
\mathcal{C}_{1}^{n}=\frac{\mathcal{B}_{1}^{n}}{n},\thinspace\thinspace\mathcal{C}_{1,2}^{n}=\frac{\left(\mathcal{C}_{1}+\mathcal{B}_{2}\right)^{n}}{n},\dots\,\,
\text{and}\,\,\,
\mathcal{C}_{1,2,\dots,k+1}^{n}=\frac{\left(\mathcal{C}_{1,2,\dots,k}+\mathcal{B}_{k+1}\right)^{n}}{n}
\]
with the symbolic computation rule: 

\textbf{$\mathcal{C}-$symbols rule}: All symbols $\mathcal{C}_{1,2,\dots,k}$
are expanded using the above identities to express them only in terms of $\mathcal{B}_{k}.$
The evaluation
rules (\ref{eq:rule1}) and (\ref{eq:rule2}) for the Bernoulli symbols
are then applied.\\
\noindent
\textbf{Example 2}.
For example,
\begin{eqnarray*}
\mathcal{C}_{1}^{n_{1}}\mathcal{C}_{2}^{n_{2}} & = & \mathcal{C}_{1}^{n_{1}}\frac{\left(\mathcal{C}_{1}+\mathcal{B}_{2}\right)^{n_{2}}}{n_{2}}=\frac{1}{n_{2}}\sum_{k=0}^{n_{2}}\binom{n_{2}}{k}\mathcal{C}_{1}^{n_{1}+k}\mathcal{B}_{2}^{n_{2}-k}
  =  \frac{1}{n_{2}}\sum_{k=0}^{n_{2}}\binom{n_{2}}{k}\frac{\mathcal{B}_{1}^{n_{1}+k}}{n_{1}+k}\mathcal{B}_{2}^{n_{2}-k}
\end{eqnarray*}
is evaluated as
\[
\frac{1}{n_{2}}\sum_{k=0}^{n_{2}}\binom{n_{2}}{k}\frac{B_{n_{1}+k}}{n_{1}+k}B_{n_{2}-k}.
\]

The next result is given in terms of this notation.
\begin{theorem}
The multiple zeta values (\ref{eq:general}) at the negative integers
$\left(-n_{1},\dots,-n_{r}\right)$ are given by
\begin{equation}
\zeta_{r}\left(-n_{1},\dots,-n_{r}\right)=\prod_{k=1}^{r}\left(-1\right)^{n_{k}}\mathcal{C}_{1,\dots,k}^{n_{k}+1}.\label{eq:zeta1}
\end{equation}
\end{theorem}
\begin{proof}
The inner sum in (\ref{eq:general}), in its Bernoulli symbols version,
\[
\sum_{l_{1},\dots,l_{r}}\binom{\bar{n}+r-\bar{k}}{l_{1}}\binom{k_{2}}{l_{2}}\dots\binom{k_{r}}{l_{r}}\mathcal{B}^{l_{1}}\dots\mathcal{B}^{l_{r}},
\]
can be summed to
\[
\left(1+\mathcal{B}_{1}\right)^{\bar{n}+r-\bar{k}}\left(1+\mathcal{B}_{2}\right)^{k_{2}}\dots\left(1+\mathcal{B}_{r}\right)^{k_{r}}.
\]
The classical identity \footnote{this identity can be deduced from the generating function 
\[
\exp\left(z\mathcal{B}\right)=\frac{z}{\exp\left(z\right)-1}.
\]
} for Bernoulli symbols $\mathcal{B}+1=-\mathcal{B},$ with $\bar{n}$ defined in \eqref{nbar} reduces this to
\begin{equation}
\left(-1\right)^{\bar{n}+1}\mathcal{B}_{1}^{\bar{n}+r-\bar{k}}\mathcal{B}_{2}^{k_{2}}\dots\mathcal{B}_{r}^{k_{r}}.\label{eq:reduced}
\end{equation}
It follows that
\begin{eqnarray*}
\zeta_{r}\left(-\mathbf{n}\right) & = & \frac{\left(-1\right)^{\bar{n}}}{\left(n_{r}+1\right)}\sum_{k_{2},\dots,k_{r}}\mathcal{C}_{1}^{\bar{n}+r-\bar{k}}\mathcal{B}_{2}^{k_{2}}\dots\mathcal{B}_{r}^{k_{r}}\prod_{j=2}^{r}\frac{\binom{\sum_{i=j}^{r}n_{i}+r-j+1-\sum_{i=j+1}^{n}k_{i}}{k_{j}}}{\left(\sum_{i=j}^{r}n_{i}+r-j+1-\sum_{i=j+1}^{n}k_{i}\right)}.
\end{eqnarray*}
Summing first over $k_{2}$ gives
\[
\zeta_{r}\left(-\mathbf{n}\right)=\frac{\left(-1\right)^{\bar{n}}}{\left(n_{r}+1\right)}\sum_{k_{3},\dots,k_{r}}\mathcal{C}_{1}^{n_{1}+1}\mathcal{C}_{2}^{n_{2}+\dots+n_{r}+r-1}\mathcal{B}_{3}^{k_{3}}\dots\mathcal{B}_{r}^{k_{r}}\prod_{j=3}^{r}\frac{\binom{\sum_{i=j}^{r}n_{i}+r-j+1-\sum_{i=j+1}^{n}k_{i}}{k_{j}}}{\left(\sum_{i=j}^{r}n_{i}+r-j+1-\sum_{i=j+1}^{n}k_{i}\right)}.
\]
The result now follows by summing, in order, over the remaining indices.
\end{proof}

Observe that the reduction (\ref{eq:reduced}) performed in the proof allows to
restate a simpler version of Sadaoui's formula (\ref{eq:general})
as the more tractable $\left(r-1\right)-$fold sum
\begin{eqnarray}
\zeta_{r}\left(-n_{1},\dots,-n_{r}\right) & = & \left(-1\right)^{\bar{n}}\sum_{k_{2},\dots,k_{r}}\frac{1}{\left(\bar{n}+r-\bar{k}\right)}\prod_{j=2}^{r}\frac{\binom{\sum_{i=j}^{r}n_{i}+r-j+1-\sum_{i=j+1}^{n}k_{i}}{k_{j}}B_{l_{1}}\dots B_{l_{r}}}{\left(\sum_{i=j}^{r}n_{i}+r-j+1-\sum_{i=j+1}^{n}k_{i}\right)}.\label{eq:general-2}
\end{eqnarray}

Observe moreover that the derivation of (\ref{eq:zeta1})
is unchanged if the symbols $\mathcal{B}_{1},\dots,\mathcal{B}_{r}$
are replaced by a generalization of the Bernoulli symbol $\mathcal{B},$
namely the polynomial Bernoulli symbol $\mathcal{B}+z$ defined by
\[
\left(\mathcal{B}+z\right)^{n}=B_{n}\left(z\right),
\]
the Bernoulli polynomial of degree $n.$ The same proof as above yields the next statement.\\
\begin{theorem}
The analytic continuation of the zeta function as given in \cite{Sadaoui} can be written as
\begin{equation}
\label{eq:zeta pol}
\zeta_{r}\left(-n_{1},\dots,-n_{r},z_{1},\dots,z_{r}\right) =\prod_{i=1}^{r} \mathcal{C}_{1,\dots,i}^{n_{i}+1}\left(z_{1},\dots,z_{i}\right)
\end{equation}
with
\[
\mathcal{C}_{1}^{n}\left(z_{1}\right)=\frac{\left(z_{1}+\mathcal{B}_{1}\right)^{n}}{n}=\frac{B_{n}\left(z_{1}\right)}{n},\thinspace\thinspace\mathcal{C}_{1,2}^{n}\left(z_{1},z_{2}\right)=\frac{\left(\mathcal{C}_{1}\left(z_{1}\right)+\mathcal{B}_{2}+z_{2}\right)^{n}}{n},\dots
\]
and
\[
\mathcal{C}_{1,2,\dots,k+1}^{n}\left(z_{1},\dots,z_{k+1}\right)=\frac{\left(\mathcal{C}_{1,2,\dots,k}\left(z_{1},\dots,z_{k}\right)+\mathcal{B}_{k+1}+z_{k+1}\right)^{n}}{n}.
\]

\end{theorem}

\section{A general recursion formula on the depth}

The methods above are now used to produce a general recursion formula on the depth of the zeta function.
\begin{theorem}
The multiple zeta functions satisfy the recursion rule
\begin{equation}
\label{recursion}
\zeta_{r}\left(-\mathbf{n};\mathbf{z}\right) =  \frac{\left(-1\right)^{n_{r}}}{n_{r}+1}\sum_{k=0}^{n_{r}+1}\binom{n_{r}+1}{k}\left(-1\right)^{k}\zeta_{r-1}\left(-n_{1},\dots,-n_{r-1}-k;\mathbf{z}\right)B_{n_{r}+1-k}\left(z_{r}\right).
\end{equation}
Introducing the new zeta symbol $\mathcal{Z}_{r}$ with the evaluation
rule \footnote{note that $\mathcal{Z}_{r}^{0}\ne1$}
\[
\mathcal{Z}_{r}^{k}=\zeta_{r}\left(-n_{1},\dots,-n_{r-1},-n_{r}-k;\mathbf{z}\right),
\]
this recursion rule can be  written symbolically as
\begin{equation}
\label{eq:recursion}
\zeta_{r}\left(-\mathbf{n};\mathbf{z}\right)=\left(-1\right)^{n_{r}}\frac{\left(\mathcal{B}-\mathcal{Z}_{r-1}\right)^{n_{r}+1}}{n_{r}+1}=\zeta_{1}\left(-n_{r};-\mathcal{Z}_{r-1}\right).
\end{equation}
\end{theorem}
\begin{proof}
Start from (\ref{eq:zeta pol}) and expand the last term
\[
\mathcal{C}_{1,\dots,r}^{n_{r}+1}\left(z_{1},\dots,z_{r}\right)=\frac{\left(\mathcal{C}_{1,\dots,r-1}^{n_{r-1}+1}\left(z_{1},\dots,z_{r-1}\right)+\mathcal{B}_{r}\left(z_{r}\right)\right)^{n_{r}+1}}{n_{r}+1}
\]
by using the binomial formula to produce
\begin{eqnarray*}
\zeta_{r}\left(-n_{1},\dots,-n_{r},z_{1},\dots,z_{r}\right) & = & \frac{\left(-1\right)^{n_{r}}}{n_{r+1}}\sum_{k=0}^{n_{r}+1}\binom{n_{r}+1}{k}\left(\prod_{i=1}^{r-2}\mathcal{C}_{1,\dots,i}^{n_{i}+1}\left(z_{1},\dots,z_{i}\right)\right)\\
 & \times & \mathcal{C}_{1,\dots,r-1}^{n_{r}+1+k}\left(z_{1},\dots,z_{r-1}\right)\mathcal{B}_{r}^{n_{r}+1-k}\left(z_{r}\right).
\end{eqnarray*}
Then identify
\[
\left(\prod_{i=1}^{r-2}\mathcal{C}_{1,\dots,i}^{n_{i}+1}\left(z_{1},\dots,z_{i}\right)\right)\mathcal{C}_{1,\dots,r-1}^{n_{r}+1+k}\left(z_{1},\dots,z_{r-1}\right)
\]
as
\[
\left(-1\right)^{n_{1}+\dots+n_{r-2}+n_{r-1}+k}\zeta_{r-1}\left(-n_{1},\dots,-n_{r-2},-n_{r-1}-k;\mathbf{z}\right)
\]
to obtain the desired result. 

Using the symbol $\mathcal{Z},$ this identity can be rewritten as
\[
\zeta_{r}\left(-n_{1},\dots,-n_{r},z_{1},\dots,z_{r}\right)=\frac{\left(-1\right)^{n_{r}}}{n_{r+1}}\left(\mathcal{B}-\mathcal{Z}_{r-1}\right)^{n_{r}+1}
\]
and the initial value
\[
\zeta_{1}\left(-n;z\right)=\left(-1\right)^{n}\frac{\left(z+\mathcal{B}\right)^{n+1}}{n+1}
\]
provides the stated recursion.
\end{proof}

\section{Contiguity identities}

The multiple zeta function at negative integer values satisfies contiguity identities in the $z$ variables. Two of them are presented here.
\begin{theorem}
The zeta function satisfies the contiguity identity
\begin{eqnarray*}
\zeta_{r}\left(-n_{1},\dots,-n_{r};z_{1},\dots,z_{r-1},z_{r}+1\right) & = & \zeta_{r}\left(-n_{1},\dots,-n_{r};z_{1},\dots,z_{r-1},z_{r}\right)
  +  \left(-1\right)^{n_{r}}\left(z_{r}-\mathcal{Z}_{r-1}\right)^{n_{r}}.
\end{eqnarray*}

\end{theorem}
\noindent
\textbf{Example 3}.
In the case of the zeta function of depth $2$,
\begin{eqnarray*}
\zeta_{2}\left(-n_{1},-n_{2},z_{1},z_{2}+1\right)  =  \zeta_{2}\left(-n_{1},-n_{2},z_{1},z_{2}\right)
  +  \left(-1\right)^{n_{1}+1}\left(z_{2}-\mathcal{Z}_{1}\right)^{n_{2}}
\end{eqnarray*}
and the second term is expanded as
\[
\left(-1\right)^{n_{1}+1}\sum_{k=0}^{n_{2}}\binom{n_{2}}{k}z_{2}^{n_{2}-k}\left(-1\right)^{k}\zeta_{1}\left(-n_{1}-k;z_{1}\right).
\]
\\
\begin{proof}
Expand
\begin{eqnarray*}
\zeta_{r}\left(-n_{1},\dots,-n_{r};z_{1},\dots,z_{r-1},z_{r}+1\right) & = & \frac{\left(-1\right)^{\bar{n}}}{n_{r}+1}\mathcal{C}_{1}^{n_{1}+1}\left(z_{1}\right)\dots\mathcal{C}_{1,\dots,r-2}^{n_{r-2}+1}\left(z_{1},\dots,z_{r-2}\right)\\
 & \times & \sum_{k=0}^{n_{r}+1}\binom{n_{r}+1}{k}\mathcal{C}_{1,\dots,r-1}^{n_{r-1}+1+k}\left(z_{1},\dots,z_{r-1}\right)B_{n_{r}+1-k}\left(z_{r}+1\right).
\end{eqnarray*}
and use the identity on Bernoulli polynomials

\[
B_{n_{r}+1-k}\left(z_{r}+1\right)=B_{n_{r}+1-k}\left(z_{r}\right)+ \left(n_{r}-k+1\right)z_{r}^{n_{r}-k}
\]
to produce the result.
\end{proof}

The corresponding result for a shift in the first variable admits a similar proof.\\
\begin{theorem}
The depth-$2$ zeta function satisfies the contiguity identities
\[
\zeta_{2}\left(-n_{1},-n_{2},z_{1}+1,z_{2}\right)=\zeta_{2}\left(-n_{1},-n_{2},z_{1},z_{2}\right)+\frac{\left(-1\right)^{n_{1}+n_{2}}}{n_{2}+1}z_{1}^{n_{1}}B_{n_{2}+1}\left(z_{1}+z_{2}\right).
\]

\end{theorem}

\section{A Generating Function}

The generating function of the zeta values at negative integers  is defined by
\begin{equation}
F_{r}\left(w_{1},\dots,w_{r}\right)=\sum_{n_{1},\dots,n_{r}\ge0}\frac{w_{1}^{n_{1}}\dots w_{r}^{n_{r}}}{n_{1}!\dots n_{r}!}\zeta_{r}\left(-n_{1},\dots,-n_{r}\right).\label{eq:generating function}
\end{equation}
A recurrence for $F_{r}$ is presented below. The initial condition is given in terms of the generating function for Bernoulli numbers
\[
F_{B}\left(w\right)= \sum_{n = 0}^{+ \infty}\frac{B_n}{n!}z^n =\frac{w}{e^{w}-1}.
\]

\begin{theorem}
The generating function of the zeta values at negative integers  satisfies the recurrence
\[
F_{r}\left(w_{1},\dots,w_{r}\right)=
\frac{1}{w_{r}} \left[
F_{r-1}\left(w_{1},\dots,w_{r-1}\right)
- F_{B}\left(-w_{r}\right)
F_{r-1}\left(w_{1},\dots,w_{r-2},w_{r-1}+w_{r}\right) 
\right]
\]
with the initial value

\[
F_{1}\left(w_{1}\right)=-\frac{1}{w_{1}}\left[e^{-w_{1}\mathcal{B}_{1}}-1\right]=\frac{1-F_{B}\left(-w_{1}\right)}{w_{1}}.
\]
Moreover, the representation of the  shift operator as $\exp\left(a\frac{\partial}{\partial w}\right) \circ f\left(w\right) = f\left(w+a\right)$ and $F_{1}\left(w,z\right)=-\frac{1}{w}\left[e^{-w\left(\mathcal{B}+z\right)}-1\right]$ give the  recursion symbolically as
\[
F_{r}\left(w_{1},\dots,w_{r}\right)=F_{1}\left(w_{r},-\frac{\partial}{\partial w_{r-1}}\right)\circ F_{r-1}\left(w_{1},\dots,w_{r-1}\right),
\]
so that
\[
F_{r}\left(w_{1},\dots,w_{r}\right)=F_{1}\left(w_{r},-\frac{\partial}{\partial w_{r-1}}\right)\circ 
F_{1}\left(w_{r-1},-\frac{\partial}{\partial w_{r-2}}\right)\circ \dots \circ
F_{1}\left(w_{2},-\frac{\partial}{\partial w_{1}}\right)
\circ F_{1}\left(w_{1}\right)
.
\]

\end{theorem}

\begin{proof}
Start from
\begin{eqnarray*}
F_{r}\left(w_{1},\dots,w_{r}\right) &=&\sum_{n_{1},\dots n_{r}}\frac{w_{1}^{n_{1}}\cdots w_{r}^{n_{r}}}{n_{1}!\cdots n_{r}!}\left(-1\right)^{n_{1}+\cdots+n_{r}}\mathcal{C}_{1}^{n_{1}+1}\cdots\mathcal{C}_{1,\dots,r}^{n_{r}+1}=\prod_{j=1}^{r}\mathcal{C}_{1,\dots,j}e^{-w_{j}\mathcal{C}_{1,\dots,j}},
\end{eqnarray*}
and expand
\[
\mathcal{C}_{1,\dots,r}e^{-w_{r}\mathcal{C}_{1,\dots,r}} = \sum_{n=0}^{\infty}\frac{\left(-w_{r}\right)^{n}}{n!}\cdot\frac{\left(-1\right)^{n+1}}{n+1}\left(\mathcal{C}_{1,\dots,r-1}+\mathcal{B}_{r}\right)^{n+1} = -\frac{1}{w_{r}}\left(e^{-w_{r}\left(\mathcal{C}_{1,\dots,r-1}+\mathcal{B}_{r}\right)}-1\right),
\]
to deduce that $F_{r}\left(w_{1},\dots,w_{r}\right)$ is 
\begin{eqnarray*}
&&\frac{1}{w_{r}} 
\left(\prod_{j=1}^{r-1}\mathcal{C}_{1,\dots,j}e^{-w_{j}\mathcal{C}_{1,\dots,j}}\right)
-\frac{1}{w_{r}}
\left(\prod_{j=1}^{r-2}\mathcal{C}_{1,\dots,j}e^{-w_{j}\mathcal{C}_{1,\dots,j}}\right)e^{-w_{r}\mathcal{B}_{r}}\mathcal{C}_{1,\dots,r-1}e^{-\left(w_{r-1}+w_{r}\right)\mathcal{C}_{1,\dots,r-1}}\\
&=& \frac{1}{w_{r}}F_{r-1}\left(w_{1},\dots,w_{r-1}\right)-\frac{1}{w_{r}}F_{B}\left(-w_{r}\right)F_{r-1}\left(w_{1},\dots,w_{r-2},w_{r-1}+w_{r}\right).
\end{eqnarray*}
\end{proof}

\section{Shuffle Identity}
Multiple zeta values at positive integers satisfy 
\textit{shuffle identities}, such as
\[
\zeta_{2}\left(n_1,n_2\right) + \zeta_{2}\left(n_2,n_1\right)
+ \zeta_{1}\left(n_1 + n_2\right) = \zeta_{1}\left(n_1\right) \zeta_{1}\left(n_2\right).
\]
The analytic continuation technique used in \cite{Sadaoui} does not preserve this identity at negative integers, while others do (for example, see \cite{Paycha}). The following theorem gives the correction terms.\\
\begin{theorem}
The zeta values at negative integers as defined in \cite{Sadaoui} satisfy the identity
\begin{equation}
\label{shuffle}
\zeta_{2}\left(-n_1,-n_2\right) + \zeta_{2}\left(-n_2,-n_1\right)
+ \zeta_{1}\left(-n_1 - n_2\right) - \zeta_{1}\left(-n_1\right) \zeta_{1}\left(-n_2\right)
= 
\frac{(-1)^{n_{1}+1} n_{1}!n_{2}!}{\left(n_1+n_2+2\right)!}B_{n_1+n_2+2}. 
\end{equation}
\end{theorem}
\noindent
\textbf{Remark 1.}
When $n_1+n_2$ is odd, $B_{n_1+n_2+2}=0$ so that the shuffle identity \eqref{shuffle} holds for $\zeta_{2}\left(-n_1,-n_2\right)$ as expected, since the depth$-2$ zeta function is holomorphic at these points.\\

\begin{proof}
Let $\delta\left(w_1,w_2\right)=F_{2}\left(w_{1},w_{2}\right) + F_{2}\left(w_{2},w_{1}\right)
+ F_{1}\left(w_{1}+w_{2}\right) - F_{1}\left(w_{1}\right)F_{1}\left(w_{2}\right).$ An elementary calculation gives
\[
\delta\left(w_1,w_2\right)
= \frac{\frac{1}{w_{1}}+\frac{1}{w_{2}}-\frac{1}{2}\coth\left(\frac{w_{1}}{2}\right)-\frac{1}{2}\coth\left(\frac{w_{2}}{2}\right)}{w_{1}+w_{2}}.
\]
The expansions
\[
\frac{1}{w_{1}}-\frac{1}{2}\coth\left(\frac{w_{1}}{2}\right)
= -\sum_{k = 0}^{+ \infty} \frac{w_{1}^{2k+1}}{\left(2k+2\right)!}B_{2k+2}\thinspace \thinspace \text{and} \thinspace \thinspace
\frac{1}{w_1+w_2}= \frac{1}{w_2}\sum_{l \ge 0} \left(- \frac{w_1}{w_2}\right)
\] 
now produce
\[
\delta\left(w_1,w_2\right) = -\sum_{k,l = 0}^{+ \infty} \left(-1\right)^{l} \frac{B_{2k+2}}{\left(2k+2\right)!}\left(w_{1}^{2k+l+1}w_{2}^{-l-1}+w_{1}^{l}w_{2}^{2k-l}\right).
\]
Identifying the coefficient of $w_{1}^{n_{1}}w_{2}^{n_{2}}$ in this series expansion gives the result.
\end{proof}

\section{Specific multiple zeta values}

This final section gives some examples of the evaluation at negative
integers of the zeta function, obtained from (\ref{eq:general})
and (\ref{eq:recursion}).

1: for depth $r=2,$

\begin{equation}
\zeta_{2}\left(-n,0\right)=\left(-1\right)^{n}\left[\frac{B_{n+2}}{n+2}-\frac{1}{2}\frac{B_{n+1}}{n+1}\right],\label{eq:zeta2(-n,0)}
\end{equation}
and
\begin{equation}
\zeta_{2}\left(0,-n\right)=\frac{\left(-1\right)^{n+1}}{n+1}\left[B_{n+1}+B_{n+2}\right].\label{eq:zeta2(0,-n)};
\end{equation}

2: for depth $r=3,$

\begin{equation}
\zeta_{3}\left(-n,0,0\right)=\frac{\left(-1\right)^{n}}{2}\left[\frac{B_{n+3}}{n+3}-2\frac{B_{n+2}}{n+2}+\frac{2}{3}\frac{B_{n+1}}{n+1}\right]\label{eq:zeta3(-n,0,0)}
\end{equation}
and
\begin{equation}
\zeta_{3}\left(0,-n,0\right)=\frac{\left(-1\right)^{n+1}}{2}\left[\frac{n}{\left(n+1\right)\left(n+2\right)}B_{n+2}-\frac{B_{n+1}}{n+1}+2\frac{B_{n+3}}{n+2}\right].\label{eq:zeta3(0,-n,0)}
\end{equation}

3: as a final example, the recursion rule \eqref{recursion} is used to compute the value $\zeta_{3}\left(0,0,-2\right)$
as 
\begin{eqnarray*}
\zeta_{3}\left(0,0,-2\right) & = & \frac{\left(\mathcal{B}-\mathcal{Z}_{2}\right)^{3}}{3}
  =  \frac{1}{3}\left(\mathcal{B}^{3}\mathcal{Z}_{2}^{0}-3\mathcal{B}^{2}\mathcal{Z}_{2}^{1}+3\mathcal{B}\mathcal{Z}_{2}^{2}-\mathcal{Z}_{2}^{3}\right)\\
 & = & \frac{1}{3}\left(B_{3}\zeta_{2}\left(0,0\right)-3B_{2}\zeta_{2}\left(0,-1\right)+3B_{1}\zeta_{2}\left(0,-2\right)-\zeta_{2}\left(0,-3\right)\right) = -\frac{1}{60}.
\end{eqnarray*}

\section*{Acknowledgements}
The first author acknowledges the partial support of NSF-DMS
0713836. The second author is partially supported, as a graduate student, by the same grant. The work of the third author was partially funded by the iCODE Institute,
a research project of the Idex Paris-Saclay.


\begin{thebibliography}{00}
\bibitem{Sadaoui}B. Sadaoui, Multiple zeta values at the non-positive
integers, C. R. Acad. Sci. Paris, Ser. 1, 2014, 12-352, 977-984

\bibitem{Zagier}D. Zagier, Values of zeta functions and their applications,
Progress in Math., 120, 1994, 497\textendash 512

\bibitem{Akiyama}S.Akiyama S.Egami and Y.Tanigawa, An analytic continuation
of multiple zeta functions and their values at non-positive integers,\textquotedblright{}
Acta Arith. XCVIII, 2001, 107\textendash 116

\bibitem{Matsumoto}K. Matsumoto, The analytic continuation and the
asymptotic behavior of certain multiple zeta-functions I, J. Number
Theory 101, 2003, 223\textendash 243

\bibitem{Broadhurst}D. J. Broadhurst, Exploiting the 1,440-fold symmetry
of the master two-loop diagram, Zeit. Phys. C 32, 1986, 249\textendash 253

\bibitem{Paycha}D. Manchon and S. Paycha, Nested sums of symbols and renormalised multiple zeta functions, preprint, arXiv:math/0702135

\bibitem{Takamuki}T. Takamuki, The Kontsevich invariant and relations
of multiple zeta values, Kobe J. Math. 16, 1999, 27\textendash 43

\bibitem{Zhao}Zhao, J.: Analytic continuation of multiple zeta-functions.
Proc. Am. Math. Soc. 128, 2000, 1275\textendash 1283 
\end{thebibliography}
\end{document}